\documentclass[10pt]{amsart}  
\usepackage{graphicx}
\usepackage{color}
\usepackage{latexsym}
\usepackage{amssymb}                
\usepackage{amsmath}
\usepackage{amsthm}
\usepackage{amscd}
\usepackage{enumerate}
\usepackage{verbatim}
\usepackage{xy}
\xyoption{all}
\theoremstyle{plain}
\newtheorem{thm}{Theorem}[section]
\newtheorem{lem}{Lemma}[section]
\newtheorem{definition}{Definition}[section]

\newtheorem{prop}{Proposition}[section]
\newtheorem{cor}{Corollary}[section]
\newtheorem{corollary}{Corollary}[section]
\theoremstyle{remark}
\newtheorem{rem}{Remark}[section]
\newtheorem{question}{Question}[section]

\numberwithin{equation}{section}
\def\ens{\ensuremath}                                     
\def\bydef{\ens{\stackrel{\text{def}}{=}}}
\newcommand\mtb[1]{\ens{\mathbb{#1}}}                     

\newcommand{\N}{\mtb{N}}	   \newcommand{\Q}{\mtb{Q}}   	\newcommand{\R}{\mtb{R}}	
\newcommand{\T}{\mtb{T}}	   \newcommand{\Z}{\mtb{Z}}

\newcommand\mtc[1]{\ens{\mathcal{#1}}}           
\newcommand{\MM}{\mtc{P}}                	
 
\newcommand{\alp}{\ens{\alpha}}  \newcommand{\bet}{\ens{\beta}}
  \newcommand{\eps}{\ens{\varepsilon}}
	\newcommand{\lam}{\ens{\lambda}}
\newcommand\bld[1]{\ens{\boldsymbol{#1}}} 
\newcommand{\bfx}{\ens{\bld x}}     \newcommand{\bfy}{{\ens{\bld y}}}   
\newcommand{\bfu}{{\ens{\bld u}}}   \newcommand{\bfr}{{\ens{\bld r}}}
\newcommand{\bfz}{{\ens{\bld z}}}   

\newcommand\uv[2]{\scalebox{#1}{#2}} 
\newcommand\uvm[2]{\scalebox{#1}{\ens{#2}}} 
\newcommand{\card}{{\bf card}}                       
\newcommand{\hdim}{{\uvm{.95}{\rm{dim_{\uv{.65}{\rm H}}}}}}   
\newcommand{\dist}{{\bf dist}}                       

\newcommand{\fractional}[1]{\langle#1\rangle}
\newcommand\hl[1]{{\center \color{red} \ens{\bs\bs\bs} \\}} 
\newcommand\hsp[1]{\mbox{}\hspace{#1mm}} 
\newcommand\vsp[1]{\par \vspace{#1mm}} 
\newcommand\vspm[1]{\par \vspace{-#1mm}} 
\def\bs{{\bigstar}}                     
\newcommand\mbc[1]{}                                                    

\newcommand{\h}{\hsp}

\newcommand{\minus}{\!\setminus\!}

\renewcommand{\ni}{\noindent}
 
\def\B{\ens{\mtb B}} 

\def\DIS{\mathrm{DIS}}
\def\DEN{\mathrm{DEN}}

\def\e{\mathrm{E}}
\def\E{\mathrm{E}}
\def\PB{\mathrm{PB}}
\def\PUB{\mathrm{PUB}}
\def\PE{\mathrm{PE}}
\def\ND{\mathrm{ND}}
\def\BA{\mathrm{BA}}
\def\diam{\mathbf{diam}}
\def\un1{\,\cup\,}
\def\dimh{\mathrm{DIM_{h}}}
\def\hd{\,\mathrm{D_H}}   \def\dh{\,\mathrm{D_H}}
\def\NN{\mtc N}    
\title[A dichotomy for projections]
{A dichotomy for projections of planar sets}
\author[M. Boshernitzan {}]{Michael Boshernitzan {}}
\address{Department of Mathematics, Rice University, Houston, TX~77005, USA}
\email{michael@rice.edu}
\thanks{The author was supported in part by research grant: NSF-DMS-1102298}
\date{November 22, 2011}
\begin{document}
\maketitle
\begin{abstract}
We prove that most one-dimensional projections of a discrete subset of $\R^2$
are either dense in~$\R$, or form a discrete subset of $\R$. More precisely,
the set $\E$ of exceptional directions (for which the indicated dichotomy fails)
is a meager subset (of the unit circle  $\T$) of Lebesgue measure 0. 
The set $\E$  however does not need to be small in the sense of Hausdorff dimension. 
\end{abstract}

\section{Main Results.}
For $n\geq1$ and  $\bfx,\bfy\in\R^{n}$, denote by $\bfx\cdot\bfy=\sum_{k=1}^n x_ky_k$ 
the standard inner product in $\R^n$, so that $||\bfx||=\sqrt{\bfx\cdot\bfx}$, 
for $\bfx\in\R^n$. 

For $r\geq0$,  denote by $\B^{n}(r)=\{\bfx\in\R^n\mid ||\bfx||\leq r\}$  the closed 
ball of radius $r$  with the center at the origin.  A set  $M\subset\R^n$ is called 
discrete if\ the intersections  $M\cap \B^{n}(r)$ are finite for all $r>0$.

For $\alp\in\R$, denote ${\bfu}_{\alp}=(\cos \alp,\sin \alp)\in S^{1}$ where 
$S^{1}=\big\{\bfx\in\R^{2} \ \big|\ ||\bfx||=1\big\}$ stands for the unit circle. 
A point in $\bfx\in S^{1}$  is determined by its direction $\alp\in \T=[0,2\pi)$ 
(so that $\bfx=\bfu_{\alp}$).

Denote by $\Phi_\alp$ the projection map \ $\Phi_{\alp}\colon\,\R^{2}\to\R$ \
defined by the formula
\begin{equation}\label{eq:proj}
\Phi_{\alp}(\bfx)=\bfx\cdot \bfu_\alp=x_1\cos \alp+x_2\sin\alp,
   \qquad  \bfx=(x_{1},x_{2})\in\R^{2}.
\end{equation}
\begin{definition}\label{def:3sets}
\em For a set $M\subset\R^{2}$, define the following three subsets of the set\, 
$\T=[0,2\pi)$ (of directions):
\begin{enumerate}
\item[(1)] the set of {\em dense} directions for $M$:
\[
\DEN(M)=\big\{\alp\in\T\ \big|  \text{ the set }\ 
        \Phi_{\alp}(M) \text{ is dense in } \R\big\},
\]
\item[(2)] the set of {\em discrete} directions for $M$:
\[
\DIS(M)=\big\{\alp\in\T\ \big| \
        \Phi_{\alp}(M) \text{ is a discrete subset in } \R\big\},
\]
\item[(3)] the set of {\em exceptional} directions for $M$:
\[
\E(M)=\T\minus\big(\DEN(M) \un1 \DIS(M)\big).
\]
\end{enumerate}
\end{definition}
Thus, for every  $M\subset\R^{2}$,                     \mbc{eq:part}
\begin{equation}\label{eq:part}
       \T=\DEN(M) \un1 \DIS(M) \un1 \E(M)
\end{equation}
is a partition of the set  $\T$  of all directions into three distinct subsets.

The central result of the paper is given by the following theorem.             \mbc{thm:cent}
\begin{thm}\label{thm:cent}
For every discrete subset $M\subset\R^{2}$, the set of exceptional directions
$\E(M)$ is a meager subset of\, $\T$  of Lebesgue measure $0$.
\end{thm}
In other words,  the above theorem claims that ``most`` directions, in both metric 
and topological senses, are either discrete or dense.                          \mbc{rem:borel}

\begin{rem}\label{rem:borel}
One easily verifies that, for any (not necessarily discrete) subset $M\subset\R^{2}$,
the sets $\DEN(M)$, $\DIS(M)$ and $\E(M)$ introduced in Definition \ref{def:3sets}
are Borel (see Corollary \ref{cor:borel} in the next section).
\end{rem}
\ni{\bf Notation.} \ Through this paper, the following notation is used:
                       
\h3{\bf 1. }\makebox[24mm][l]{$\dimh(X)$}\h3 
             stands for the Hausdorff dimension of a set $X\in\T$;\vspm{.5}
             
\h3{\bf 2. }\makebox[24mm][l]{$\lam(X)$}\h3 
             stands for the Lebesgue measure of a set $X\in\T$;\vspm{.5}

\h3{\bf 3. }\makebox[24mm][l]{$\card(X)\leq\omega$}\h3 
             means that $X$ is at most countable.\vspm{.5}

\h3{\bf 4. }\makebox[24mm][l]{$X\triangle Y$}\h3 stands for the symmetric 
             difference $(X\minus Y)\cup(Y\minus X)$ of sets $X$ and $Y$;\vspm{.5}
           
\h3{\bf 5. }\makebox[24mm][l]{$X\equiv Y\!\pmod \omega$}\h3 means that\, 
            $\card(X\triangle Y)\leq\omega$  
            (i.\,e., $X$ and $Y$ differ by at most a countable set);\vspm{.5}
           
\h3{\bf 6. }\makebox[24mm][l]{$X\equiv Y\pmod \lam$}\h3 
           means that \ $\lam(X\triangle Y)=0$ \h{5}
           (i.\,e., $X$ and $Y$ differ by a set of measure $0$).

\begin{rem}\label{rem:hd}
We shall see that there are discrete subsets $M\subset\R^{2}$  with $\dimh(\E(M))=1$
(i.\,e., the exceptional set of direction may have full Hausdorff dimension
even though  $\lam(\E(M))=0$\, by Theorem \ref{thm:cent}). 
Some examples of such $M$ are given by Propositions   \ref{prop:difsqrs}
and \ref{prop:exmr}.
\end{rem}
Theorem \ref{thm:cent} admits a generalization for arbitrary (not necessarily discrete) 
subsets $M\subset\R^{2}$. 
This genera\-li\-zation is given by Theorem~\ref{thm:cent2} below.
We need the following definition.
\begin{definition}\label{def:3bsets}
   \em For a subset $M\subset\R^2$, define the following three sets:
\begin{itemize}
\item[(1)] the set\, $\PB(M)$ of {\em P-bounded} directions for $M$:
\begin{align*}
\PB(M)=\Big\{\alp\in\T\ \Big|\ (\Phi_\alp)^{-1}(J)\cap M\, 
    &\text{ is bounded in }\ \R^2, \\[-1mm]
&\text{ for every bounded subset } J\subset\R\Big\},
\end{align*}
\item[(2)] the set $\PUB$ of {\em P-unbounded} directions for $M$:
\begin{align*}
\PUB(M)=\Big\{\alp\in\T\ \Big|\ (\Phi_\alp)^{-1}(J)\cap M\, 
         &\text{ is unbounded in }\ \R^2,   \\[-1mm]
&\text{ for every open non-empty subset } J\subset\R\Big\},
\end{align*}
\item[(3)] the set\, $\PE$ of {\em P-exceptional} directions for $M$:
\begin{equation}\label{eq:pem}
\PE(M)=\T\minus\big(\PB(M)\un1 \PUB(M)\big).
\end{equation}
\end{itemize}
\end{definition}
\vsp1
\begin{thm}\label{thm:cent2}
For every subset  $M\subset\R^2$,  the set\, $\PE(M)$  (of $P$-exceptional directions)
is a meager set of Lebesgue measure 0.
\end{thm}
In other words, ``most'' directions, in both metric and topological senses, 
are either $P$-bounded, or \mbox{$P$-unbounded}. 

We observe that Theorem \ref{thm:cent2} indeed implies Theorem \ref{thm:cent}
in view of the following relations (taking place for discrete 
subsets \ $M\subset\R^2$):                                          \mbc{eqs:PE+E}
\begin{equation}\label{eqs:PE+E}
\E(M)\subset \PE(M), \qquad   \card\big(\PE(M)\minus\E(M)\big)\leq\omega.
\end{equation}
\vsp1
These relations (for discrete $M$) are derived easily  
from the following ones:                         \mbc{eq:PUB+DEN\\eq:PB+DIS\\eq:PB+DIS2}
\begin{subequations}
\begin{align}\label{eq:PUB+DEN}
\PUB(M)&=\DEN(M);\\[1mm] \label{eq:PB+DIS}
\PB(M)&\subset\DIS(M);\\[1mm] \label{eq:PB+DIS2}
\card\big(\DIS(M&)\minus\PB(M)\big)\leq\omega. 
\end{align}
\end{subequations}

While \eqref{eq:PUB+DEN} and \eqref{eq:PB+DIS} are obvious, 
\eqref{eq:PB+DIS2} follows from the inclusion\,
$\DIS(M)\minus\PB(M)\subset W(M)$, where $W(M)$ is the set of all directions 
determined by pairs of distinct points in  $M$, and the fact that \
$\card(W(M))\leq\omega$ \ because \ $\card(M)\leq\omega$.

Note that the partition \eqref{eq:part} is stable 
(modulo subsets of Lebesgue measure $0$)  under bounded pertur\-ba\-tions of a set 
$M\subset\R^{2}$  (Theorem~\ref{thm:sum}). 

Also, there are multidimensional analogues
of Theorems \ref{thm:cent} and \ref{thm:cent2} (see Section \ref{sec:mult}).

\section{ Exceptional sets are Borel. Maps  $\Psi_{\bet}$.}
It is often more convenient to work with the maps  $\Psi_{\bet}\colon \R^{2}\to\R$ 
defined by the formula
\begin{equation}\label{eq:proj2}
   \Psi_{\bet}(\bfx)=x_{1}+\bet x_2, \qquad \bfx=(x_1,x_2)\in\R^2,\qquad 
      (\text{for }\,\bet\in\R),
\end{equation}
rather than with the maps $\Phi_{\alp}$  (see \eqref{eq:proj}). 

By analogy with Definition \ref{def:3sets}, for every $M\subset\R^2$,  
one introduces the sets                                    \mbc{eq:DENP\\eq:DISP\\eq:EP}
\begin{subequations}
\begin{align}
\DEN'(M)&=\{\bet\in\R\mid \Psi_{\bet}(M) \ \text{ is dense}\}\label{eq:DENP},\\
\DIS'(M)&=\{\bet\in\R\mid \Psi_{\bet}(M) \ \text{ is discrete}\}\label{eq:DISP}\\[-1mm]
\intertext{and}                                              
\E'(M)&=\R\minus\big(\DIS'(M)\un1\DEN'(M)\big)\label{eq:EP}. 
\end{align}
\end{subequations}
The obvious connection
\[
\Phi_{\alp}(\bfx)=\cos \alp\cdot \Psi_{\bet}(\bfx), \qquad  \bet=\tan \alp;
\]
implies the equalities    \mbc{eq:DENPM\\eq:DISPM\\eq:E+E}
\begin{subequations}
\begin{align}
\DEN'(M)&=\tan(\DEN(M));\label{eq:DENPM} \\ 
\DIS'(M)&=\tan(\DIS(M));\label{eq:DISPM} \\[-1mm]
\intertext{and}  
\E'(M)&=\tan(\E(M)). \label{eq:E+E}
\end{align}
\end{subequations}

In view of \eqref{eq:DENPM}--\eqref{eq:E+E}, the facts that the sets $\DEN(M)$, 
$\DIS(M)$ and $\E(M)$  are Borel (for an arbitrary subset $M\subset\R^2$) 
follow immediately from the following theorem.               \mbc{thm:borelp}
\begin{thm}\label{thm:borelp}
For every set $M\subset\R^2$, the sets\,  $\DEN'(M), \DIS'(M), \E'(M)$ 
are Borel  subsets of\, $\R$.
\end{thm}
\begin{cor}\label{cor:borel}
For every set $M\subset\R^2$, the sets\,  $\DEN(M)$, $\DIS(M)$, $\E(M)$ are 
Borel subsets of\, $\R$.
\end{cor}
\begin{proof}[Proof of Corollary \ref{cor:borel}]
Follows from Theorem \ref{thm:borelp} and \eqref{eq:DENPM}--\eqref{eq:E+E}.
\\
\end{proof}
\begin{proof}[Proof of Theorem \ref{thm:borelp}]
Denote by \ $\Sigma$ \ the family of all rational subintervals of $\R$ (i.\,e., 
non-empty subintervals of $\R$ with the rational endpoints).  The presentation
\[
\DEN'(M)=\bigcap_{J\in\Sigma}\ \big\{\bet\in\R\mid\Psi_{\bet}(M)\cap
     J\neq\emptyset\big\}
\]
shows that $\DEN'(M)$  is Borel and in fact a $G_{\delta}$-set 
(a countable intersection of open sets).

We assume without loss of generality that $\card(M)\leq\omega$. (Otherwise $M$  is 
replaced by any of its at most countable dense subset; this replacement does not affect 
the sets \ $\DEN'(M)$,  $\DIS'(M)$ and $\E'(M)$). 

Denote by $W(M)$ the set of $\bet\in \R$  for which the map \
$\Psi_{\bet}\big|_{M}\colon M\to\R$ \ fails to be injective. 
Then \ $\card\big(W(M)\big)\leq\omega$ \ because \ $\card(M)\leq\omega$,
and the equation  $\Psi_{\bet}(p)=\Psi_{\bet}(q)$  has at most one solution $\bet\in\R$, 
for any pair of distinct points \ $p,q\in M$.  

Arrange the countable set $M$  into a 
sequence  $M=\{m_k\mid k\geq1\}$.  One verifies that \\
\[
\DIS'(M)\minus W(M)=\{\bet\in\R\mid\Psi_{\bet}(M)\cap J 
    \text{ is finite, for all } J\in\Sigma\}
    =\bigcap_{\substack{N\geq1\\[.5mm] J\in\Sigma }}
         \Big(\bigcup_{k\geq N} U(k,J)\Big)
\]
\vsp0
\ni where all the sets \  $U(k,J)=\big\{\bet\in\R\mid\Psi_{\bet}(m_k)\in J\big\}$ \ 
are open.  Thus  \ $\DIS'(M)\minus W(M)$ \ is Borel. Since 
$\card(W(M))\leq\omega$,  $\DIS'(M)$ is also Borel. 
Finally, $\E'(M)$ is Borel in view of \eqref{eq:EP}.\vspm2

\end{proof}
\section{Proof of Theorem \ref{thm:cent}}
In view of \eqref{eq:DENPM}--\eqref{eq:E+E},  it is enough to prove
the following theorem.
\begin{thm}\label{thm:ep}
For every discrete subset $M\subset\R^{2}$, the set $\E'(M)$ (defined by 
\eqref{eq:E+E}) is a meager subset of\, $\T$  of Lebesgue measure $0$.
\end{thm}
\begin{proof}
Denote by \ $\Sigma$ \ the family of all rational subintervals of $\R$ (i.\,e., 
non-empty subintervals of $\R$ with the rational endpoints).
For any two finite subintervals  $P,Q\in\Sigma$,  define the set
\begin{equation}\label{eq:vpq}
V(P,Q)=\Big\{\bet\in\R\ \Big|\ \Psi_{\bet}(M)\cap P \text{ is infinite, and } 
    \Psi_{\bet}(M)\cap Q=\emptyset \Big\}.
\end{equation}
For $\beta\in\R$, the condition $\bet\in\E'(M)$  is equivalent to the existence of 
two finite interval  $P$ and $Q$  (without loss of generality, $P,Q\in\Sigma$)  such 
that $\bet\in V(P,Q)$. (Indeed, the existence of $P$  means that $\bet\notin\DIS'(M)$,
and the existence of $Q$ is equivalent to the condition $\bet\notin\DEN'(M)$).

Thus  $\E'(M)$  can be represented as the countable union
\[
\E'(M)=\bigcup_{P,Q\in\Sigma}V(P,Q).
\]
To complete the proof of Theorem \ref{thm:ep}, it remain to verify that every set
$V(P,Q)$  is nowhere dense and has Lebesgue measure $0$.

Fix  $\bet\in V(P,Q)$, $\bet\neq0$. Since $\Psi_{\beta}(M)\cap P$  is infinite, 
there exists an infinite sequence of distinct points  
$\bfz_k=(x_k,y_k)\in M\subset\R^2$, $k\geq1$,  
such that  $\Psi_{\bet}(\bfz_k)=x_k+\bet y_k\in P$. Since  $M$  is discrete,
\[
\lim_{k\to\infty}||\bfz_k||\to\infty.
\]
Moreover, we have
\[
\lim_{k\to\infty}|x_k|=\lim_{k\to\infty}|y_k|=\infty
\]
because the points $\Psi_{\bet}(\bfz_k)=x_k+\bet y_k\in P$\ lie in the
(bounded) interval $P$, and $\bet\neq0$.

We may assume that all $y_k\neq0$ (by dropping a few first terms of $\{\bfz_k\}$ 
if needed).  Denote $\eps_k=|y_k|^{-1}$. Consider two sequences of intervals:
\begin{align*}
P_k&=\Big\{t\in\R\ \big| \ \Psi_t(\bfz_k)=x_k+ty_k\subset P\Big\}=(P-x_k)\,\eps_k\\
Q_k&=\Big\{t\in\R\ \big| \ \Psi_t(\bfz_k)=x_k+ty_k\subset Q\Big\}=(Q-x_k)\,\eps_k.
\end{align*}
Set
\begin{align*}
d&=\diam(P\cup Q); && q\ =\diam(Q)\h2 =\lam(Q);\\
d_k&=\diam(P_k\cup Q_k); && q_{k}=\diam(Q_{k})=\lam(Q_{k});
\end{align*}
where $\diam(A)$ stands for the diameter of a set $A$. Clearly \
$q_k=q\cdot\eps_k$ \ and \  $d_k=d\cdot\eps_k$.

Since $\bet\in V(P,Q)$, for all $k\geq1$, the relations
\[
\bet\in P_{k}\quad \text{ and } \quad Q_{k}\cap V(P,Q)=\emptyset
\]
hold by the definition of $V(P,Q)$  (see \eqref{eq:vpq}). 

We observe that, for every $k\geq1$, the $d_k$-neighborhood 
\[
J_k=\big(\bet-d_k, \bet+d_k\big)
\] 
of $\bet$ contains a subinterval $Q_k$ of length  $q_k$  which does not intersect 
$V(P,Q)$. Note that $\lam(J_k)=2d_k\to 0$ as $k\to 0$, and the ratio \ 
$\frac{\lam(Q_k)}{\lam(J_k)}=\frac q{2d}\leq\frac12$ \ does not depend on $k$.

It has been shown that for every $\bet\in V(P,Q)\minus\{0\}$ \ one can find
arbitrary small intervals $J_k$  around $\beta$  which contains a further subinterval 
$Q_k$ of relative density $\frac q{2d}$ such that $Q_k\cap V(P,Q)=\emptyset$.

The above property implies that the set $V(P,Q)$  is nowhere dense and has Lebesgue
measure $0$  (because it has no Lebesgue density points). The proof of 
Theorem \ref{thm:ep} is complete. \\[-2mm]
\end{proof}
\section{Exceptional set $\E(M)$ may have full Hausdorff dimension}    \label{sec:exhd1}

The following proposition provides an example of a discrete set $M\subset\R^2$ 
whose exceptional set  $\E(M)$  has full Hausdorff dimension.
A class of examples of such $M$  will be presented in the next section.   \mbc{prop:difsqrs}
\begin{prop}\label{prop:difsqrs}
For the set
\[
M_0=\{(m^2,n^2)\mid m,n\in\N={1,2,3,\ldots}\}\subset\R^2,
\]
the exceptional set has full Hausdorff dimension: \ $\hdim(\E(M_0))=1$.
\end{prop}
In view of \eqref{eq:E+E}, it is enough to show that\, $\hdim(\E'(M_0))=1$.

For $\alp\in\R$,  denote                                        \mbc{eq:Lalp}
\begin{equation}\label{eq:Lalp}
L(\alp)=\Psi_{\alp}(M_0)\subset\R=\big\{\alp m^2-n^2\mid m,n\in\N\big\}.
\end{equation}
Then (see notation in \eqref{eq:DENP}--\eqref{eq:EP})
\begin{subequations}
\begin{align}
\DEN'(M_0)&=\{\alp\in\R\mid L(\alp) \text{ is dense in }\R\},\label{eq:DENP0}\\
\DIS'(M_0)&=\{\alp\in\R\mid L(\alp) \text{ is discrete in }\R\},\notag\\[-1mm]
\intertext{and}
\E'(M_0)&=\R\minus\big(\DIS'(M_0)\un1\DEN'(M_0)\big).\label{eq:EP0}
\end{align}
\end{subequations}
\vsp1
\begin{lem}\label{lem:noDIS}
No irrational\, $\alp>0$ lies in $\DIS'(M_0)$.
\end{lem}
\begin{proof}
Assume to the contrary that some irrational $\alp\in\DIS'(M_0)$. 
Then the number $\bet=\sqrt{\alp}$  is also irrational,
and the diophantine inequality \ $|m\bet-n|<\frac1m$ \ has infinitely 
many solutions in \ $m,n\in\Z^2, m\geq1$. For each such a solution, we have
\begin{align*}
\big|\alp m^2-n^2\big|&=\big|m\bet-n\big|\cdot\big|m\bet+n\big|<
       \frac{\left|m\bet+n\right|}m=\\
       &=\frac{\big|2m\bet-(m\bet-n)\big|}m\ <\,2\bet+1.
\end{align*}
Sinve $\alp$ is irratinal,  we conclude that  $L(\alp)$ contains infinitely many
points
in the finite interval  \mbox{$(-(2\bet+1),\,2\bet+1)$}.  Thus $L(\alp)$  is not discrete,
a contradiction with the assumption that   $\alp\in\DIS'(M_0)$.\\
\end{proof}
For $x\in\R$,  denote by $\fractional{x}$ the distance from $x$ to the closest integer:
$\fractional{x}=\dist(x,\Z)=\min_{k\in\Z}\limits|x-k|$. 

Let $\N=\{k\in\Z\mid k\geq1\}=\{1,2,\ldots\}$ stand for the set of natural numbers.
A number $\alp\in\R$ is called
{\em badly approximable} if there exists an $\eps>0$ such that 
$m\fractional{m\alp}>\eps$, for all $m\in\N$. 
Denote by $\BA$ the set of badly approximable numbers. It is clear that this set
does not contain rationals: \ $\BA\cap\Q=\emptyset$.
\begin{lem}\label{lem:noDEN}
If $\bet\in\BA$ then $\alp=\bet^2\notin\DEN'(M_0)$.
\end{lem}
\begin{proof}
Since  $\bet\in\BA$, there exists an $\eps>0$ such that $m\fractional{m\alp}>\eps$,
for all $m\in\N$. Without loss of generality, one assumes that $\bet>0$. Then, for any
$m,n\in\N$, the following inequality holds:
\[
\big|\alp m^2-n^2\big|=\big|m\bet-n\big|\cdot\big|m\bet+n\big|\geq
      \fractional{m\bet}\cdot\big|m\bet+n\big|>\frac {\eps\,|m\bet+n|}m>\eps\bet.
\]

We observe that $L(\alp)\cap\big(-\eps\bet,\eps\bet\big)=\emptyset$ (see \eqref{eq:Lalp}),
whence  $\alp\notin\DEN'(M_0)$ (by \eqref{eq:DENP0}), a contradiction.\\
\end{proof}
\begin{corollary}\label{cor:sqrtE}
Assume that $\alp>0$ is irrational such that $\bet=\sqrt \alp\in\BA$  
(i.\,e., that $\bet$ is badly approximable). Then $\alp\in\E'(M_0)$.
\end{corollary}
\begin{proof}
One derives $\alp\notin\DIS'(M_0)$ and $\alp\notin\DEN'(M)$ from 
Lemmas \ref{lem:noDIS} and \ref{lem:noDEN}, respectively. Therefore
$\alp\in\E'(M_0)$, in view of \eqref{eq:EP0}.\\[-2mm]
\end{proof}
Denote by  \ $C=(\BA)^2\minus\Q$ \ the set of irrational squares of badly 
approximable numbers. By Corollary~\ref{cor:sqrtE},   $C\subset\E'(M_0)$. It is 
well known that $\hdim(BA)=1$. (The set $\BA\subset\R$ of badly approximable 
numbers has full Hausdorff dimension, see e.\,g.~\cite{Schmidt}). 

It follows that  $1\geq\hdim(\E'(M_0))\geq\hdim(C)=1$ and hence
$\hdim(\E(M_0))=1$ (in view of \eqref{eq:E+E}).  This completes the proof
of Proposition \ref{prop:difsqrs}.
\begin{rem}
The above arguments coupled with Theorem \ref{thm:cent} provide a short proof 
of the known fact that the set  $\BA$  of badly approximable numbers has 
Lebesgue measure $0$.    
\end{rem}  
\section{More examples} \label{sec:examples}
A sequence  $\bfr=\{r_k\}_{k\geq1}$  of real numbers is said to be 
{\em rising to infinity}  if it is strictly increasing and if 
$\lim_{k\to\infty}\limits r_k=\infty$. For any such a sequence\, $\bfr$, 
define the discrete set     \mbc{eq:mr}
\begin{equation}\label{eq:mr}
    M(\bfr)=\big\{(n,r_k)\mid n,k\in\Z, 
       k\geq1\big\}\subset\R^2.
\end{equation}
By Theorem \ref{thm:cent}, $\lam(\e(M(\bfr)))=0$.                 \mbc{prop:exmr}
\begin{prop}\label{prop:exmr}
  Let   $\bfr=\{r_k\}_{k\geq1}$ be a rising to infinity sequence 
  of real numbers. Then
\begin{itemize}
\item[(1)] If\, $\bfr$ is lacunary (i.\,e., if \ 
           $\liminf_{k\geq1}\limits\frac{r_{k+1}}{r_k}>1$)
           then\, $\dimh(\E(M(\bfr)))=1$.
\item[(2)] If\, $\bfr$ is sublacunary (i.\,e., if \ 
    $\lim_{k\to\infty}\limits\frac{r_{k+1}}{r_k}=1$) then\, $\dimh(\E(M(\bfr)))=0$.
\item[(3)] If\, $\limsup_{k\to\infty}\limits\, (r_{k+1}-r_k)<\infty$\, then\, 
   $\card(\E(M(\bfr)))\leq\omega$\,  (i.\,e., the set\, $\E(M(\bfr))$  
   is at most countable).
\end{itemize}
\end{prop}

In what follows in this section, we assume that  $\bfr=\{r_k\}_{k\geq1}$ be a rising 
to infinity sequence.  One easily verifies that\, 
$\PB(M(\bfr))=\emptyset$, and hence (see  \eqref{eq:PB+DIS2})
\begin{equation}\label{eq:disw}
\card(\DIS(M(\bfr)))\leq\omega.
\end{equation}

The following proposition follows immediately from Theorem \ref{thm:cent}.
Recall that a set is called residual if its complement is meager.
\begin{prop}\label{prop:A}
$\DEN(M(\bfr))\subset\T=[0,2\pi)$ is a residual set of full Lebesgue measure (in $\T$).
\end{prop}
Taking in account \eqref{eq:DENPM}, we conclude the following.
\begin{corollary}\label{cor:1}
$\DEN'(M(\bfr))\subset\R$ is a residual set of full Lebesgue measure (in $\R$).
\end{corollary}
On the other hand,  the set  $\DEN'(M(\bfr))$  may be defined in the following way:
\begin{align}
  \DEN'(M(\bfr))&=\Big\{\bet\in\R\mid\text{ the set } \big\{(n+\bet r_k)\mid n,k\in\Z, 
     k\geq1\big\}\  \text{ is dense in }\R\Big\}=\notag \\
     &=\Big\{\bet\in\R\mid\text{ the sequence } \big\{\bet r_k\big\} 
      \text{ is dense}\pmod1\Big\} \label{al:DEN},
\end{align}
Now we are ready to derive the following (known) result.
\begin{prop}\label{prop:dense}
If a set  $S\subset\R$ is unbounded, then the set                         \mbc{eq:ND}   
\[
\ND(S)=\big\{\bet\in\R\ \big|\ \bet\cdot S 
     \text{ is not dense}\h{-2}\pmod1\big\}
\]
is a meager subset of\, $\R$ of Lebesgue measure $0$.
\end{prop}
Proposition \ref{prop:dense} follows easily from classical uniform distribution results 
(see e.\,g. \cite[Ch.1, \S4, Cor.~4.3]{KN}). For a direct simple proof
see \cite[\S6]{Bosh_sublac}.   What follows is a derivation of 
Proposition \ref{prop:dense} from Theorem \ref{thm:cent}.                                  
\begin{proof}[\bf Proof of Proposition \ref{prop:dense}]
Let assume for definiteness that the set $S$ is unbounded from above.  
Then there exists a rising to infinity sequence $\bfr=\{r_k\}_{k\geq1}$  of 
positive reals in $S$.    

Let $M(\bfr)$ be defined as in \eqref{eq:mr}. 
Since the set\, $\ND(S)$ is a subset of the complement 
of the set\,  $\DEN'(M(\bfr))$ in \eqref{al:DEN}, the claim of the proposition 
follows from Corollary \ref{cor:1}. \\[-2mm]
\end{proof}
\begin{proof}[\bf Proof of Proposition \ref{prop:exmr}]
We demonstrate that Proposition \ref{prop:exmr} is just a reformulation of 
some known results on distribution mod $1$ of certain sequences of reals.
Denote 
\[
\ND(\bfr)=\Big\{\bet\in\R\ \big|\ \text{ the sequence } 
     \big\{\bet\cdot r_{k}\big\}\ \text{ is not dense}\h{-2}\pmod1\Big\}.
\]
Since $\ND(\bfr)$ is the complement of the set  $\DEN'(M(\bfr))$ in $\R$,
we obtain
\[
\E'(M(\bfr))\subset\ND(\bfr),\qquad  \ND(\bfr)\minus\E'(M(\bfr))=\DIS'(M(\bfr)).
\]
where\, $\card(\DIS'(M(\bfr)))\leq\omega$,  in view of \eqref{eq:disw}
and \eqref{eq:DISPM}.

We conclude that
\begin{subequations}
\begin{equation}\label{eq:dimh=}
\dimh(\ND(\bfr))=\dimh(\E'(M(\bfr)))=\dimh(\E(M(\bfr))),
\end{equation}
and that (see \eqref{eq:E+E})
\begin{equation}\label{eq:om}
\card(\ND(\bfr))\leq\omega \h1 \implies \h1 \card(\E'(M(\bfr)))\leq\omega
        \h1  \implies \h1 \card(\E(M(\bfr)))\leq\omega.
\end{equation}
\end{subequations}

It is known (\cite{M1}, \cite{M2}, \cite{P}) that if\, $\bfr=\{r_k\}_{k\geq1}$
is a lacunary sequence of positive integers, then the set  $\ND(\bfr)$ has full
Hausdorff dimension: $\dimh(\ND(\bfr))=1$.
The claim (1) of Proposition \ref{prop:exmr}  now follows from~\eqref{eq:dimh=}.

On the other hand, by \cite[Theorem 1.3]{Bosh_sublac} \  $\dimh(\ND(\bfr))=0$ \
for sublacinary sequences $\bfr=\{r_k\}_{k\geq1}$  rising to infinity. 
The claim (2) of Proposition \ref{prop:exmr}  also follows from~\eqref{eq:dimh=}.

Finally, the claim (3) of Proposition \ref{prop:exmr} follows from \eqref{eq:om}
and the fact that, under the current assumption that \
$\limsup_{k\to\infty}\limits\, (r_{k+1}-r_k)<\infty$, 
the set $\ND(\bfr)$ must be at most countable 
\cite[Proposition 1.8]{Bosh_sublac}. (In the special case when all  
$r_k$  are integers the fact is also proved in 
\mbox{\cite[Corollaries 42 and 43]{R}}). \\[-3mm]
\end{proof}
\section{Projections of arbitrary sets. 
          Proof of Theorem \ref{thm:cent2}}        \mbc{sec:prthm:cent2} 
The relations \eqref{eq:PUB+DEN}--\eqref{eq:PB+DIS2} mean, in particular, 
that, for discrete $M\subset\R^2$,  the partition \eqref{eq:part} coincides 
(modulo countable sets) with the partition      \mbc{eq:part2}
\begin{equation}\label{eq:part2}
   \T=\PUB(M)\un1\PB(M)\un1 \PE(M)
\end{equation}

The proof of Theorem \ref{thm:cent2} (that the set $\PE(M)$ must be small in both
metric and topological senses) is based on Theorems \ref{thm:cent} and 
Theorem \ref{thm:hmPB} which asserts stability of $\PB(M)$  under bounded 
perturbations of subsets  $M\subset\R^2$.

For $r>0$ and a set  $A\subset \R^k$,  we use the standard
notation
\begin{equation}\label{eq:neigh}
\NN(r,A)\bydef\{\bfx\in\R^k\mid \dist(A, \bfx)<r\},
\end{equation}
for the $r$-neighborhood of $A$  where  \ $\dist(A, \bfx)\bydef\inf_{a\in A}\limits ||x-a||\in[0,\infty]$.

For two subsets  $A,B\subset\R^{2}$  define \mbc{eq:hd0}
\begin{equation}\label{eq:hd0}
D_0(A,B)=\sup_{b\in B}\,\dist(A,b)=
   \inf \Big(\big\{r>0\mid B\subset\NN(r,A)\big\}\Big)\in\big[0,\infty\big].
\end{equation}

Recall that the {\em Hausdorff distance} $\hd(M_1,M_2)$ between two non-empty sets
 $M_1,M_2\in\R^2$ is defined as            \mbc{eq:hd\\thm:hmPB\\down}
\begin{equation}\label{eq:hd}
  \hd(M_1,M_2)=\max\big(D_0(M_1,M_2),\,D_0(M_2,M_1)\big)
  \in\big[0,\infty\big].
\end{equation}
\begin{thm}\label{thm:hmPB}
Assume that  $\hd(M_1,M_2)<\infty$ for two  subsets  $M_1,M_2\subset\R^2$. Then
\[
\PB(M_1)=\PB(M_2).
\]
\end{thm}
Theorem \ref{thm:hmPB}  follows immediately from the following     \mbc{thm:hmPB-up\\prop:hm<}
\begin{prop}\label{prop:hm<}
Assume that for some subsets  $M_1,M_2\subset\R^2$
\[
D_0(M_1,M_2)=\sup_{x\in M_2} \dist(x,M_1)<\infty.
\]
Then\,  $\PB(M_1)\subset\PB(M_2)$.
\end{prop}
\begin{proof}[Proof of Proposition \ref{prop:hm<}] 
Denote $u=D_0(M_1,M_2)$ and select any $v>u\geq0$.
Then $M_2\subset \NN(v,M_1)$.

Given $\alp\in\PB(M_1)$, we have to show that  $\alp\in\PB(M_2)$. This is to say that,
for every bounded interval $J=(a,b)\subset\R$, the set  
$K=(\Phi_{\alp})^{-1}(J)\cap M_2$  is bounded in $\R$.

Denote by $J'$  the open interval  $J'=\NN(v,J)=(a-v,b+v)$.
Then we have  (see notation \eqref{eq:neigh})
\begin{align*}
K\,=\,(\Phi_{\alp})^{-1}(J)\cap M_2&\subset\,(\Phi_{\alp})^{-1}(J)\cap
   \NN(v,M_1)\subset \\
   &\subset\,\NN(v,(\Phi_{\alp})^{-1}(J')\cap M_1)\bydef L.
\end{align*}
Since $\alp\in\PB(M_1)$, the set  $P=(\Phi_\alp)^{-1}(J')\cap M_1$  is bounded.
It follows that the sets $L=\NN(v,P)$ and $K\subset L$ are also bounded, 
completing the proof of Proposition \ref{prop:hm<}.\\[-2mm]
\end{proof}
\begin{proof}[Proof of Theorem \ref{thm:hmPB}]
Since $\dh(M_1,M_2)<\infty$, both sets $D_0(M_1,M_2)$ and $D_0(M_2,M_1)$
are finite. Apply Proposition \ref{prop:hm<}.\\[-2mm]
\end{proof}
\begin{proof}[Proof of Theorem \ref{thm:cent2}]
There exists a subset $M'\subset M$ such that $M\subset\NN(1,M')$ and 
$M'$  is discrete in~$\R^2$. (Take a subset $M'\subset M$ which is maximal 
under the constraint that the distances between any distinct points of $M'$  
are\,  $\leq1$. 

Let $U=\DIS(M')\minus\PB(M')$. In view of \eqref{eq:PB+DIS} and \eqref{eq:PB+DIS2}, we have
\[
\card(U)\leq\omega \qquad \text{\small and} \qquad \PB(M')=\DIS(M')\minus U.
\]
On the other hand, it follows from \eqref{eq:PUB+DEN} and the inclusion 
$M'\subset M$ that 
\[
\DEN(M')=\PUB(M')\subset \PUB(M).
\]
We obtain (see \eqref{eq:pem}):
\[
\PE(M)=\T\minus\big(\PB(M)\un1 \PUB(M)\big)\subset
       \T\minus\Big(\big(\DIS(M)\minus U\big)\un1 \DEN(M')\Big)
       \subset\E(M')\cup U.
\]
The claim of Theorem \ref{thm:cent2} follows from Theorem \ref{thm:cent} 
and the fact that  $\card(U)\leq\omega$. \\[-1mm]
\end{proof}
The following theorem summarizes the results on stability of the partitions 
\eqref{eq:part}  and  \eqref{eq:part2} (see notation following Remark \ref{rem:hd}),
under bounded perturbations of a set $M$.
Recall that \ $\hd(\,\cdot\,,\,\cdot\,)$ \ stands for the Haudorff distance between sets,
see \eqref{eq:hd}.

\begin{thm}[Summary]\label{thm:sum}
Assume that  $\hd(M_1,M_2)<\infty$,  for subsets  $M_1,M_2\subset\R^2$.
Then:

\h8{\bf (1) }$\PB(M_1)=\PB(M_2)$;

\h8{\bf (2) }$\lam(\PE(M_i))=0$, for both $i=1,2$;

\h8{\bf (3) }$\PUB(M_1)\equiv\PUB(M_2) \pmod \lam$;

If, moreover, both $M_1,M_2$ are discrete then

\h8{\bf (4) }$\DIS(M_1)\equiv\DIS(M_2)\equiv\PB(M_2) \pmod \omega$;

\h8{\bf (5) }$\DEN(M_1)\equiv\DEN(M_2)\equiv\PUB(M_2) \pmod \lam$.
\end{thm}
\begin{proof}[Proof of Theorem \ref{thm:sum}]
{\bf (1)} and {\bf (2)} are exactly Theorems \ref{thm:hmPB} and \ref{thm:cent2}, 
respectively. {\bf (3)} follows from {\bf (1)}. Finally, {\bf (4)} follows from
{\bf (1)}, \eqref{eq:PB+DIS} and \eqref{eq:PB+DIS2},  and {\bf (5)} follows from
{\bf (3)} and \eqref{eq:PUB+DEN}.\\[-2mm]
\end{proof}
We derive the following corollary for projections of syndetic subset of $\R^2$.
A set $M\subset\R^2$ is called {\em syndetic}\, if  $\dh(M,\R^2)=D_0(M,\R^2)<\infty$
(see \eqref{eq:hd0} and \eqref{eq:hd}).          \mbc{prop:syn}
\begin{prop}\label{prop:syn}
For a syndetic subset $M\subset\R^2$,  one has
\[
\PUB(M)\equiv\DEN(M)\equiv \T=[0,2\pi).
\]
\end{prop}
\begin{proof}
By Theorem \ref{thm:sum}, {\bf (3) }, $\T=\PUB(\R^{2})\equiv\PUB(M) \pmod \lam$
whence $\DEN(M)=\T\pmod\lam$, in view of \eqref{eq:PUB+DEN}.
(We may assume that $M$ is discrete because otherwise one replaces $M$ with its syndetic
discrete subset.\\[-2mm]
\end{proof}
\begin{rem}
One easily verifies that, for syndetic subsets  $M\subset\R^2$, one has  
$\PB(M)=\emptyset$, and that $\card(\DIS(M)\leq\omega$. There are examples of 
discrete sundetic subsets  $M\subset\R^2$ for which $\hd(E(M)=1$. On the other hand,  
one can prove that, for every discrete sundetic subsets  $M\subset\R^2$, 
the exceptional set  $\E(M)$  is a countable union of sets of the box dimension $<1$. 
(In fact,  then  $\E(M)$ must be a countable union of perforated sets in the sense of
\cite[\S3]{Bosh_sublac}). Under the additional assumption on $M$ to be periodic
(there exists $\bfx\in\R^2\minus\{0\}$  such that $M+\bfx=M$), $\card(\E(M))\leq\omega$. 
\end{rem}
\section{Some questions}\label{sec:q}

Denote $\R^+=[0,\infty)$. Given a continuous function $g\colon \R^+\to\R^+$, a subset
$M\subset\R^2$ is said to be $g$-syndetic if the set
\[
{\bf F}_g=\big\{\bfx\in\R^2\mid \dist(\bfx,M)\geq g(|\bfx|)\big\}
\]
is bounded. If the relation $\DEN(M)\equiv\T\pmod\lam$  holds for all $g$-syndetic
discrete subsets  $M\subset\R^2$, then the function $g$  is called  P-negligible.
By Proposition \ref{prop:syn}, bounded $g$ must be   P-negligible.
\begin{question}
Does there exist  P-negligible function  $g\colon \R^+\to\R^+$  such that
$\lim_{x\to\infty}\limits g(x)=\infty$?
\end{question}
One can speculate that $g(x)=\sqrt x$ \ is   P-negligible. This is in agreement with
fact that the set  
\[
M=\big\{(\pm m^2,\pm n^2)\mid m,n\in\N\big\}
\]
satifies $\DEN(M)\equiv\T\pmod\lam$ (cf. Proposition \ref{prop:difsqrs} and its proof in 
Section~\ref{sec:exhd1}). On the other hand, the functions  $g(x)=x^{a}$  with $a>1/2$  
fail to be P-negligible as it follows from the following proposition.
\begin{prop}\label{prop:apm}
For $a>0$, denote \ $M(a)=\big\{ (\pm m^a, \pm n^a)\mid m,n\in\N \big\}\subset\R^2$.
Then
\[
\DIS(M(a))\equiv\T\pmod\lam, \qquad \text{for }\, u>2,
\]
and 
\[
\DEN(M(a))\equiv\T\pmod\lam, \qquad \text{for }\, u\leq2.
\]
\end{prop}
The proof of Theorem \ref{prop:apm} is based on Theorem \ref{thm:cent} and the fact
that, for Lebsgue almost all $t\in\R$, the inequality\,  $m^{a-1}\fractional{mt}<1$ 
has a finite or infinite number of solutions $m\in\Z$  depending on whether or 
not  $a>2$  (cf. proof of Lemma \ref{lem:noDIS} in Section \ref{sec:exhd1}).

\begin{definition}[Notation]
Denote by $\MM$  the family of measurable subsets  $A\subset\T$  for which there 
exists a set  $M\subset\R^2$  such that \ $\DEN(M)\equiv A\pmod\lam$ \ (see 
notation following Remark \ref{rem:hd}).
\end{definition}

Note that requiring sets  $M$ to be discrete (in the above definition) would not 
affect the defined family $\MM$ because of {\bf (5)} in Theorem \ref{thm:sum}).

The problem of characterization of sets in the family $\MM$  is open. 
Clearly, a set $A\in\MM$  must be Lebesgue measurable and  $\pi$-periodic 
(which means $A+\pi\equiv A\pmod\lam)$. 

\begin{question}
Does the family $\MM$ coincide with the family of all  $\pi$-periodic measurable
subsets of $T$?
\end{question}
We observe that any $\pi$-periodic finite union $A$ of subintervals of $\T$
must lie in $\MM$. One just takes  
\[
M_1=\Big\{(x_1,x_2)\in\R^2\ \Big|\ x_1\neq0, \h2 \frac{x_2}{x_1}\in A\Big\}
\]
(or  $M_2=M_1\cap(\Z\times\Z)$  to make the set $M$ discrete).

\section{Multidimensional extensions}\label{sec:mult}
Main results of the paper (Theorems \ref{thm:cent}, \ref{thm:cent2} and \ref{thm:sum})
extend to all dimensions $n\geq2$. The following is a multidimensional version of 
Theorem \ref{thm:cent}.
\begin{thm}\label{thm:mult1}
Let  $n>k\geq1$ be integers and let $M\subset\R^n$  be a discrete subset. Then,
for Lebesgue almost all \ $k$-planes \  $U\subset\R^n$ 
(in the sense of the natural\, $k(n-k)$-Lebesgue measure
on the Grassmannian\, {\bf GR}$(k,n)$)
  the projection of\, $M$ on\, $U$
is either dense in\, $U$, or discrete in\, $U$.
\end{thm}
The proof of Theorem \ref{thm:mult1} is more complicate and longer than that of 
Theorem \ref{thm:cent},  and it is not included (even though the basic idea is the same).
It will be published elsewhere in the case an application of Theorem~\ref{thm:mult1} 
justifying the length of the proof will be found.

The statements of the multidimensional versions of Theorems \ref{thm:cent2} and \ref{thm:sum}
are straightforward, and we omit these.

\section{Projection of random sets in $\R^2$} We conclude the paper by formulating 
results (also without proofs) on the generic size of the sets $\DIS(M)$, $\DEN(M)$
and $\E(M)$, for random subset  $M\subset\R^2$  of given density. 
We consider two settings.

\subsection{First setting} Denote by $\boldsymbol\alp$ the sequence $\{\alp_k\}$ of 
independent random variables, each uniformly distributed in $\T=[0,2\pi)$.
For every increasing to infinity sequence $\bfr=\{r_k\}_{k\geq1}$\,  of positive numbers,
consider the random set\,  $M=\{\bfx_k\mid k\geq1\}$ of points in\, $\R^2$ where
\[
\bfx_k=(r_k\cos \alp_k,r_k\sin \alp_k)\in\R^2
\]
is the point with polar coordinates $(r_k, \alp_k)$. Thus the points $\bfx_k$, 
$k\geq1$, are selected independently, each point  $\bfx_k$  being picked up 
randomly on the circle\,  $\big\{\bfx\in\R^2\ \big|\ ||x||=r_k\big\}$.

The following theorem describes the generic size of the sets\, $\DIS(M)$, $\DEN(M)$ 
and $\E(M)$\,  for the discrete random set\,  $M=\{\bfx_k\mid k\geq1\}\in\R^2$. 
A statement is said to be satisfied a.\,s. (almost sure) if it holds for almost all
choices of the sequence $\boldsymbol \alp$.
\begin{thm}
Given an increasing to infinity sequence $\bfr=\{r_k\}_{k\geq1}$\, of positive numbers,  
the following statements take place:
\begin{enumerate}
\item[\em(1)] $\DEN(M)$ is a residual subset of\, $\T$, a.\,s.\\
\item[\em(2)] The relation \ $\DEN(M)=\T\pmod\lam$ a.\,s. takes place 
 if and only if \ $\sum_{k\geq1}\frac1{r_{k}} =\infty$.\\

\item[\em(3)] The relation \ $\DIS(M)=\T\pmod\lam$ a.\,s. takes place 
 if and only if \ $\sum_{k\geq1}\frac1{r_{k}} <\infty$.\\

\item[\em(4)] The relation \ $\DIS(M)=\T$ a.\,s. takes place 
 if and only if \ $\limsup_{N\to\infty}\limits\,
              \Big(\sum_{k=1}^N\limits\,\dfrac1{r_{k}\cdot\,\log N}\Big)=\infty$.
              
\end{enumerate} 
\end{thm}
\begin{proof}
Statesments (2) and (3)  are obtained by standard application of Borel-Cantelli 
lemma. One also verifies that $\DEN(M)$ is a residual subset of\, $\T=[0,2\pi)$ 
whenever $\boldsymbol \alp$ is dense in $\T$, an a.\,s. condition. This proves (1).

Statement (4) is more delicate. Some readers may find it surprising that the 
conditions on sequence $\bfr=\{r_k\}$ for (2) and (4) are not equivalent. The 
situation here is reminiscent to Dvoretzky's problem on covering the circle by
random arcs (see \cite[Ch. 11]{Ka}, \cite{Shepp} for the description of the problem
and its solution by L.\ Shepp).
Statement (4) can be derived from the solution of Dvoretzky's problem.\\[-5mm]

\end{proof}
\section{Concluding remarks}

The main result of the paper (Theorem \ref{thm:cent}) was inspired by a conversation
with Hillel Furstenberg in~1993.  
A shorter version of the present work was circulating
as an unpublished preprint of 1994.

I would like to thank Benjy Weiss and Hillel Furstenberg for useful discussions
(conducted years ago) and also Yuval Peres for his encouragement to publish 
the results of this paper.
(His 2000 paper \cite{PS}, joint with  Boris Solomyak, refers to the
unpublished preprint mentioned above).

\vsp{22}


\begin{thebibliography}{xxx}
\bibitem{Bosh_sublac}  M.\ Boshernitzan, 
   Density modulo 1 of dilations of sublacunary sequences, 
   \textit{Advvances in  Math.}\,\textbf{108} (1994),
     104--117.
\bibitem{ET} P.\ Erd\"os, S.\ J.\ Taylor, 
On the set of points of convergence of a 
lacunary trigonometric series and the equidistribution
properties of related sequences,
\textit{Proc. London Math. Soc.} (3) \textbf{7},
1957, 598--615.
\bibitem{Falc} K.\ Falconer, 
  \textbf{Fractal Geometry, Mathematical 
  Foundations and Applications}, 
  John Wiley \& Sons, Chichester, 1990.
\bibitem{Ka} J.\ P.\ Kahane, 
  \textbf{Some Random Series of Functions}, 
  Cambridge, London, New York, 
  Cambridge University Press, 1985.
\bibitem{KN} L.\ Kuiper and H.\ Niederreiter,
  \textbf{Uniform Distribution of Sequences}, 
  New York, Wiley-Interscience, 1974.
\bibitem{M1} B.\ de Mathan, 
\textit{Sur un probleme de densite modulo $1$},
C.\ R.\ Acad.\ Sc.\ Paris Series A, {\bf 287},
1978, 277--279.
\bibitem{M2} B.\ de Mathan, 
\textit{Numbers contrvening a condition on 
density modulo $1$},   Acta Math.\ Acad.\ Sci.\ 
Hungar.\ {\bf 36}, 1980, 237--241.
\bibitem{PS} Y.\ Peres, B. Solomyak,
\textit{Approximations with polynomials with 
coefficients $\pm1$}, 
J.\ Number Theory {\bf 84}, 2000, 185-198
\bibitem{P} A.\ D.\ Pollington,
\textit{On the density of sequence $n_{k}\xi$}, 
Acta Math.\ Acad.\ Sci.\ 
Hungar.\ {\bf 36}, 1980, 237--241.
\bibitem{R} J.\ Rosenblatt,
\textit{Norm convergence in ergodic theory and
the behavior of Fourier transform}, 
Can.\ J.\ Math.\ {\bf 46}\, (1), 1994, 184--199.
\bibitem{Schmidt}  W.\ M.\ Schmidt,
On Badly Approximable Numbers and Certain Games,
\textit{Transactions of the American Mathematical Society}, 
(1) \textbf{123}, 1966, 178-199.
\bibitem{Shepp} L.\ A.\ Shepp,
\textit{Covering the circle with random arcs}, 
Israel J.\ Math.\ {\bf 11}, 1972, 328--345.
\end{thebibliography}
\end{document}